\documentclass[reqno]{amsart}
\usepackage{amsmath,amssymb,amsfonts,caption}
\usepackage{graphicx}
\usepackage{enumerate}
\usepackage{verbatim}
\usepackage{epsfig}
\usepackage{enumitem}

\usepackage{color}
\usepackage[usenames,dvipsnames,svgnames,table]{xcolor}
\definecolor{orange}{rgb}{1,0.5,0}
\newtheorem{theorem}{Theorem}
\newtheorem{definition}{Definition}
\newtheorem{remark}{Remark}
\newtheorem{lemma}{Lemma}
\newtheorem{proposition}{Proposition}
\newtheorem{example}{Example}
\newtheorem{corollary}{Corollary}
\renewcommand{\epsilon}{\varepsilon}
\renewcommand{\phi}{\varphi}

\renewcommand{\ge}{\geq}

\newcommand{\x}{\times}
\DeclareMathOperator{\Id}{Id}
\DeclareMathOperator{\e}{e}

\DeclareMathOperator{\dist}{dist}

\newcommand{\w}{\omega}
\newcommand{\W}{\Omega}
\usepackage{dsfont}
   
   \newcommand{\R}{\ensuremath{\mathds R}}

\usepackage{color}
\usepackage[usenames,dvipsnames,svgnames,table]{xcolor}
\definecolor{orange}{rgb}{1,0.5,0}

\begin{document}
\title[]
   {Random perturbations of an\\ eco-epidemiological model}
\author{Lopo F. de Jesus}
\address{L. F. de Jesus
   Departamento de Matem\'atica\\
   Universidade da Beira Interior\\
   6201-001 Covilh\~a\\
   Portugal}
   \email{lopo.jesus@ubi.pt}
\author{C\'esar M. Silva}
\address{C. M. Silva\\
   Departamento de Matem\'atica\\
   Universidade da Beira Interior\\
   6201-001 Covilh\~a\\
   Portugal}
\email{csilva@ubi.pt}
\author{Helder Vilarinho}
\address{Helder Vilarinho\\
   Departamento de Matem\'atica\\
   Universidade da Beira Interior\\
   6201-001 Covilh\~a\\
   Portugal}
\email{helder@ubi.pt}
\date{\today}
\thanks{L. F. de Jesus, C. M. Silva and H. Vilarinho were partially supported by FCT through CMA-UBI (project UIDB/MAT/00212/2020).}

\begin{abstract}
We consider random perturbations of a general eco-epidemiological model. We prove the existence of a global random attractor, the persistence of susceptibles preys and provide conditions for the simultaneous extinction of infectives and predators.
We also discuss the dynamics of the corresponding random   epidemiological $SI$ and predator-prey models. We obtain for this cases a global random attractor, prove the prevalence of susceptibles/preys and provide conditions for the extinctions of infectives/predators.
\end{abstract}
\maketitle
\section{Introduction}
The understanding of asymptotic behavior of eco-epidemiological models is an important problem in the mathematical biology. However, from the very beginning of the theory it became clear that even  models for a ultra simplified version of real world phenomena exhibit intricate and complex behaviours, despite their simple formulation.
In view of this, several mathematical tools were developed with great success in order to understand as much as possible the features and properties of these models.

Motivated, in one hand, by the attempts to approximate the mathematical models to real world phenomena as much as possible and, on the other hand, by the mathematical challenge to provide a deep and general knowledge on the theory constructed due to this formulations, the models and the related theory have been reconstructed and evolved in many directions. One of these directions aimed to consider nonautonomous elements in the mathematical models, such as the seasonal dynamics, and also random elements in order to deal with the presence of noise or complicated fluctuations, in contrast to a completely deterministic situation.

In the nondeterministic situation thee are two main approaches to incorporate randomness by considering stochastic and random perturbations, which, roughly speaking, can be expressed throughout stochastic and random differential equations. For an exposition on this subjects, their comparison and application in this biological context we refer to~\cite
{Arnold, Caraballo-Colicci-Han-2016, caraballo-colucci-lopez-de-la-cruz-Rapaport-2019, Caraballo-Han-2016, Caraballo-Colicci-2017, Han-Kloeden-2017} and references therein. There are techniques to transform a system with stochastic perturbation to a random dynamical system, being that this can lead to unbounded random coefficients on the system and can substantially change the structure of the model; see for instance~\cite{Caraballo-Han-2016,Caraballo-Colicci-2017}.

Random attractors are a central concept in the analysis of random models. Since their introduction there are several improvements regarding the existence and properties of such attractors, but there are questions that  are still open in this theory; see~\cite{crauel-1999, crauel-2001, Crauel-Flandoli-1994, Crauel-Kloeden-2015, Crauel-Scheutzow-2018, Han-Kloeden-2017}. The main strategy adopted to ensure the existence of a random attractor for a given family of random sets is to find a compact absorbing set. Moreover, since the family of random sets we are interested contains every compact deterministic set, the random attractor is actually unique (cf. Remark~\ref{unique}).

In this work we consider random perturbations of a general eco-epidemiological model introduced in~\cite{jesus-silva-vilarinho-pre}, that generalizes the model in~\cite{Niu-Zhang-Teng-AMM-2011} by adding a general function corresponding to predation on uninfected and infected preys. We introduce a random coefficient, establish a framework of Random Dynamical Systems (RDS) and discuss the asymptotic behaviour of the solutions of the model considered. Namely, we focus on the existence of a global random attractor which can be understood as a random counterpart of a deterministic global attractor.  We moreover prove the prevalence of susceptibles preys and provide conditions for the simultaneous extinction of infectives and predators.

We also discuss the dynamics of the corresponding random   epidemiological $SI$ and predator-prey models, by considering the infectives ($I$) and predators ($P$) identically equal to zero in the main model~\eqref{Eq. Principal random}. For both subsystems we obtain a global random attractor, prove the persistence of susceptibles/preys and provide conditions for the extinctions of infectives/predators. Random perturbations for a $SI$ model, albeit slightly different, were discussed in~\cite{Caraballo-Han-2016,Caraballo-Colicci-2017}.

This work is organized as follows: in \S\ref{section:PSPS} we recall basic facts from RDS and random attractors; in \S\ref{section:model} we introduce a random perturbation in an eco-epidemiological model and establish a RDS framework; in \S\ref{section:random_attractors} we prove the existence of a unique global random attractor and provide a threshold for the extinction of predators and infected preys; in \S\ref{section:partial} we discuss the partial dynamics of the perturbed model in the absence of predators or infected preys.

\section{Random attractors}\label{section:PSPS}
We start by recalling some basic concepts about Random Dynamical Systems (RDS) and random attractors. For details on RDS we refer to the reference monograph by Arnold~\cite{Arnold} and for random attractors see e.g. the survey \cite{Crauel-Kloeden-2015} and reference therein.

Let $(\Omega,\mathcal{F},\mathbb{P})$ be a probability space, where $\mathcal{F}$ is the $\sigma$-algebra of measurable subsets of $\Omega$ and $\mathbb{P}$ is a probability measure on $\mathcal{F}$. Given a topological space $S$ we denote by $\mathcal B(S)$ the Borel $\sigma$-algebra of $S$. Consider a \emph{metric dynamical system} $(\Omega, \mathcal{F}, \mathbb{P}, \theta)$ in the sense that
\begin{enumerate}[label=(\roman*)]
   \item $\theta \colon \mathbb{R} \times \Omega \to \Omega$ is $({\mathcal{B} (\mathbb{R}) \otimes \mathcal{F} ,\mathcal{F}})$-measurable;
   \item $\theta_t \colon \Omega \to \Omega$ given by $\theta_t\w=\theta(t,\w)$ satisfies:
   \begin{enumerate}
   \item  $\theta_0 = \Id_{\Omega}$ and $\theta_{t+s}=\theta_t\circ\theta_s$, for
       all $t,s \in\mathbb{R}$;
   \item $\mathbb{P}(\theta_{t} A) = \mathbb{P}(A)$, for all $A \in \mathcal{F}$, that is, $\theta_t$ preserves the probability measure $\mathbb{P}$ for all $t \in \mathbb{R}$.
   \end{enumerate}
\end{enumerate}
The non-intuitive term \emph{metric} is present in the literature for historical reasons.

 A (\emph{measurable}) \emph{random dynamical system} (RDS) $ \phi $ on $X=\R^d$ over
$ \theta $ (with time $\mathbb{R}_0^+$) is a map
\begin{equation*}
  \phi:\mathbb{R}_0^+\times\Omega\times X \to X
\end{equation*}
satisfying
\begin{enumerate}[label=(\roman*)]
   \item measurability: $(t,\omega,x) \mapsto \phi(t,\omega,x)$ is $({\mathcal{B}(\mathbb{R}_0^+) \otimes \mathcal{F}\otimes\mathcal{B}(X),\mathcal{B}(X)})$-measurable;
   \item cocycle property: $\phi(t,\omega,x)$ forms a  cocycle over $\theta$, i.e.,
\begin{enumerate}
    \item $\phi(0,\omega,x)=x$, for all $\omega\in\Omega, x \in X$;
    \item $\phi(t+s,\omega,x)=\phi(t,\theta_s\omega, \phi(s,\omega,x))$, for all $s, t\in\mathbb{R}_0^+$, $\omega\in\Omega$ and $x \in X$.
\end{enumerate}
\end{enumerate}
We, moreover, assume a continuity condition:
\begin{enumerate}[label=(\roman*)]
  \setcounter{enumi}{2}
  \item $ x \to \varphi(t,\omega,x) $ is continuous for all $\omega \in \Omega$ and $t \in \mathbb{R}_0^+$.
\end{enumerate}
To simplify we refer to such a RDS as the pair $(\theta,\varphi)$.

\begin{remark}
  \begin{enumerate}
  \item[]
    \item The cocycle property is assumed to hold for all $\w\in\W$, or at least in a $\mathbb P$-full measure subset. This can be a delicate issue to ensure in specific examples of RDS generated from a stochastic or a random differential equation.
    \item Often joint continuity $(t,x) \to \varphi(t,\omega,x)$ is assumed, but the continuity in time $t$ will not have a role in the following theory of random attractors. Nevertheless, the RDS to be considered here is induced by random differential equations, which provide joint continuity in time and space. In this case, the mapping $\varphi(t,\w,x_0)$ corresponds to the solution mapping with noise realization $\w$ and initial condition $x_0=\varphi(0,\w,x_0)$; cf. Theorem~\ref{1A}.
    \item We consider the phase space $X=\R^d$ because we have in mind specific RDS related to random versions of eco-epidemiological models that evolve on $\R^3$. However, it is typical to assume $X$ to be a Polish space, i.e., a separable topological space for which there is a complete metric which induces the topology. This, in particular, includes open non-empty subsets of Euclidean spaces as well as separable Hilbert and Banach spaces.
  \end{enumerate}
\end{remark}

\begin{definition}
A \emph{random set}~$C$ is a measurable subset of $X\times\Omega$ with respect to the product $\sigma$-algebra $\mathcal B(X)\otimes\mathcal F$.
\end{definition}

 Given $\omega\in\Omega$, the $\omega$-section of a random set~$C\subseteq X\times \Omega$ is defined by
 \begin{equation*}
   C(\omega)=\{x:(x,\omega)\in C\}.
 \end{equation*}
If a set~$C\subseteq X\times\Omega$ has closed or
compact $\omega$-sections $C(\w)$ it is a random set as soon as the mapping
$\omega\mapsto d\bigl(x,C(\omega)\bigr)$ is measurable (from $\W$ to $[0,\infty[$) for every $x\in X$ (see~\cite{Castaing-Valadier-1977}).
In this case~$C$ will be said to be a \emph{closed} or a \emph{compact}
 random set, respectively. We say that a random set $C$ has deterministic components or, for short, that is a deterministic set (as subset of $X\x\W$) if its $\w$-sections are constant: there is $\hat C\subseteq X$ such that $C(\w)=\hat C$ for all (or, at least, almost all) $\w\in\W$. We define $\mathcal D(X)$ as the set of all deterministic compact random sets.  For any set~$C\subseteq X\times\Omega$, we define
$\overline{C}:=\{(x,\omega):\,x\in \overline{C(\omega)}\}$. We say that a random set $C(\w)$ is \emph{bounded} if $C(\w)\subseteq X$ is bounded for all (or, at least, almost all) $\w\in\W$.

\begin{remark}
In general, having
  $\omega\mapsto d\bigl(x,C(\omega)\bigr)$  measurable
  for every $x\in X$, does not guarantee that $C\subseteq X\x \W$ is a $(\mathcal B(X)\otimes\mathcal F)$- measurable set;
  see~\cite[Remark~4]{Crauel-Kloeden-2015}.
\end{remark}

\begin{definition}
A bounded random set $K$ is said to be \emph{tempered} with respect to $\theta$ if for $\mathbb{P}$-a.e. $\omega \in \Omega$,
$$\lim_{t \to \infty}e^{-\beta t} \sup\limits_{x \in K(\theta_{-t}\w)}\|x\|=0,  \,\text{ for all } \beta >0.$$
A random variable $r\colon\W\to\R$ is said to be \emph{tempered} with respect to $\theta$ if for $\mathbb{P}$-a.e. $\omega \in \Omega$,
$$\lim_{t \to \infty}e^{-\beta t} \sup\limits_{t \in \R} |r(\theta_t\w)\\|=0,  \,\text{ for all } \beta >0.$$
\end{definition}
We denote by $\mathcal{T}(X)$ the set of all tempered sets of $X$ (i.e, tempered bounded random sets with fibers on $X$) with respect to $\theta$. Notice that $\mathcal{D}(X)\subseteq \mathcal{T}(X)$. In our perspective, the underlying dynamics $\theta$ is given, so that we will often omit the reference to $\theta$.

\begin{definition}
Consider a RDS $(\theta,\varphi)$ on $X$ and an arbitrary family $\mathcal R$ of random sets. A random set $\Gamma$ is called a \emph{random absorbing set} in $\mathcal{R}$ if for any $K \in \mathcal{R}$ and $\mathbb{P}$-a.e. $\omega \in \Omega$, there exists $T_K(\omega)>0$ such that
$$\varphi(t,\theta_{-t}\omega,K(\theta_{-t}\omega)) \subseteq \Gamma(\omega)\,\text{ for all } t \ge T_K(\omega).$$
If, in addition, $\Gamma$ is a closed random set, then we say that $\Gamma$ is a \emph{closed absorbing set}.
\end{definition}

\begin{definition}
Consider a RDS $(\theta,\varphi)$ on $X$ and an arbitrary family $\mathcal R$ of random sets. A compact random set $\mathcal{A}$ is called a \emph{pullback $\mathcal{R}$ attractor} if:
\begin{enumerate}[label=(\roman*)]
\item invariance: for $\mathbb{P}-$a.e. $\omega \in \Omega$ and all $t \geq 0$ it holds
$$\varphi(t,\omega,\mathcal{A}(\omega))=\mathcal{A}(\theta_t\omega);$$
\item attracting property: for any $K \in \mathcal{R}$ and $ \mathbb{P} $-a.e. $\omega \in \Omega,$
\begin{equation}\label{eq:back attr}
\lim_{t \to \infty}\dist(\varphi(t,\theta_{-t}\omega,K(\theta_{-t}\omega)),\mathcal{A}(\omega))=0,\\
\end{equation}
where
\begin{equation}
dist(G,H)=\sup_{g \in G} \inf_{h \in H} \|g-h\|
\end{equation}
is the Hausdorff semi-metric for $G,H \subseteq X.$
\end{enumerate}
If $\mathcal{R}=\mathcal{T}(X)$  we say in this conditions that $\mathcal{A}$ is a \emph{global random attractor}.
\end{definition}

\begin{remark}\label{TD}
  Notice that a {global random attractor} is also a {pullback $\mathcal{D}(X)$ attractor}.
\end{remark}

\begin{remark}
Unless stated otherwise, a random attractor will be understood always as a pullback attractor. There are other notions of attraction in this context, such as forward attraction, weak attraction and attraction in probability but it is not our purpose to investigate those behaviours. Notice also that the notion of pullback attractor does not depend on the choice of the metric $dist$, which is not the case when we consider a forward attractor $\mathcal A$ which instead condition~\eqref{eq:back attr} satisfies
\begin{equation*}
\lim_{t \to \infty}\dist(\varphi(t,\omega,K(\omega)),\mathcal{A}(\theta_t\omega))=0;\\
\end{equation*}
see~\cite[Section 5]{Crauel-Scheutzow-2018}.
\end{remark}

\begin{proposition}\cite[Proposition 1]{Caraballo-Colicci-2017}\label{PropA}
Consider a RDS $(\theta,\varphi)$ on $X$ and an arbitrary family $\mathcal R$ of random sets containing $\mathcal{D}(X)$. If there is a compact absorbing set $\Gamma \in \mathcal{R}$ then there is a unique global random attractor $\mathcal{A}$  with component subsets
\begin{equation}
\mathcal{A}(\omega)=\bigcap_{\tau\geq T_\Gamma(\w)} \overline{\bigcup_{t\geq \tau} \varphi(t,\theta_{-t}\omega,\Gamma(\theta_{-t}\omega))}.
\end{equation}
If the pullback absorbing set is positively invariant, i.e., $\varphi(t,\omega,\Gamma(\omega)) \subset \Gamma(\theta_t\omega)$ for all $t \geq 0$, then
\begin{equation}
\mathcal{A}(\omega)=\bigcup_{t> T_\Gamma(\w)} \overline{\varphi(t,\theta_{-t}\omega,\Gamma(\theta_{-t}\omega))}.
\end{equation}
\end{proposition}

\begin{remark}\label{Rm1}
The original statement requires an asymptotic compactness property which is trivially satisfied in our context since $X=R^d$. Note also that this path-wise attracting in the pullback sense does not need to be path-wise attracting in the forward sense, although it is forward attracting in probability: for any $\epsilon>0$,
\begin{equation*}\begin{split}
 & \mathbb P(\{\w\in\W\colon\dist(\varphi(t,\theta_{-t}\omega,K(\theta_{-t}\omega)),\mathcal{A}(\omega))\geq\epsilon\})\\
   & = \mathbb P(\{\w\in\W\colon\dist(\varphi(t,\omega,K(\omega)),\mathcal{A}(\theta_t\omega))\geq \epsilon\})
\end{split}\end{equation*}
which goes to 0 as $t\to\infty$. That is,
\begin{equation*}  \label{tq11}
    \lim_{t\to\infty}d\bigl(\varphi(t,\omega)K(\omega),
    A(\theta_t\omega)\bigr)=0\quad \mbox{in probability}
  \end{equation*}
  for every $K\in\mathcal R$. In particular, this allows individual realizations along
sample paths to have large deviations from the attractor, but still to converge in this
probabilistic sense.
\end{remark}

\begin{remark}\label{unique}
The attractor need not be unique for a general family $\mathcal R$. However, as soon as $\mathcal R$ contains every compact deterministic set, if a random attractor for $\mathcal{R}$ exists then it is unique (cf. \cite{Crauel-Scheutzow-2018}). Notice this is the case if  $\mathcal R \in\{\mathcal D(X),\mathcal T(X)\}$.
\end{remark}

\section{A predator-prey system with disease and real noise in prey}\label{section:model}
In this section we consider random perturbations of an eco-epidemiological model discussed in~\cite{jesus-silva-vilarinho-pre}, that generalizes a model in~\cite{Niu-Zhang-Teng-AMM-2011}. This model contains a general functions corresponding to predation on uninfected and infected preys populations.
\subsection{A random eco-epidemiological model}
We consider an eco-epidemiological model in which we have two populations, the prey population ($S$) and the predators ($P$), with the prey population infected with and infectious disease ($I$). We consider a general function corresponding to the predation on uninfected prey and also to the vital dynamics of predator population and a random  birth rate of the prey population, modelled by a random variable as follows:
\begin{equation} \label{Eq. Principal random}
\begin{cases}
S'(t,\omega)=\Lambda(\theta_t\omega)-\mu S(t)-f(S(t),I(t),P(t))P(t)-\beta S(t)I(t)\\
I'(t,\omega)=\beta S(t)I(t)-\eta g(S(t),I(t),P(t))I(t)-c I(t)\\
P'(t,\omega)=\gamma f(S(t),I(t),P(t))P(t)+r \eta g(S(t),I(t),P(t))I(t)-\delta_1 P(t)-\delta_2 P(t)^2
\end{cases}
\end{equation}
where $(\Omega, \mathcal{F}, \mathbb{P}, \theta)$ is a {metric dynamical system} that drives the noise and:
\begin{enumerate}[label=(H\arabic*)]
\item $\mu, \beta,r,c,\gamma,\delta_1$ and $\delta_2$ are all positive constants, and we assume that $ \mu <c$;
\item functions $ f,g:(\mathbb{R}_0^+)^3 \to \mathbb{R}_0^+ $ are locally  Lipschitz and satisfy
\begin{enumerate}[label=(\roman*)]
\item $S \rightarrow f(S,I,P)$ and $P\mapsto g(S,I,P)$ are nondecreasing,
\item $I \rightarrow f(S,I;P)$, $P \rightarrow f(S,I,P)$, $S\mapsto g(S,I,P)$ and $I\mapsto g(S,I;P)$ are nonincreasing,
\item $f(0,I,P)=0$ and $g(S,I,0)=0$;
\item $f(S,0,0)>0$ whenever $S>0$;
\end{enumerate}\label{cond:fg}
\item $\Lambda: \Omega \rightarrow \mathbb{R}^+$ is a measurable function such that
\begin{equation}\label{Lambda}
\Lambda(\omega) \in [\Lambda^\ell, \Lambda^u]:= q_0[1-\epsilon;1+\epsilon],
\end{equation}
with $q_0>0, \epsilon \in (0,1)$, for all $\, \omega \in \Omega$, and such that the function $t \mapsto \Lambda(\theta_t\omega)$ is continuous.\label{H3}
\end{enumerate}
Typically, the functional response of the predator to prey is given by some particular function. In this paper we take as reference the model proposed in~\cite{jesus-silva-vilarinho-pre} that generalizes the one in~\cite{Lu-Wang-Liu-DCDS-B-2018} by considering general functions corresponding to the predation of uninfected and infected prey. Besides the population compartments, given by $S$, $I$ and $P$ that correspond, respectively, to the susceptible prey, infected prey and predator, we may understood $\Lambda$ and $\mu$ as the (random) recruitment rate and the natural death rate of prey population, respectively, $\beta$ as the incidence rate of the disease, $\eta$ as the predation rate of infected prey, $c$ as the death rate in the infective class, $\gamma$ as the
rate converting susceptible prey into predator (biomass transfer), $r$ as the rate of converting infected prey into predator, $f(S,I,P)$ is the predation of susceptible prey and $g(S,I,P)$ is the predation of infected prey. It is assumed that only susceptible preys $S$ are capable of reproducing, i.e, the infected prey is removed by death (including natural and disease-related death) or by predation before having the possibility of reproducing.

Note that our setting includes several of the most common functional responses for both functions $f$ and $g$: $f(S,I,P)=kS$ and $g(S,I,P)=kP$ (Holling-type I), $f(S,I,P)=kS/(1+m(S+I))$ and $g(S,I,P)=kP/(1+m(S+I))$ (Holling-type II), $f(S,I,P)=kS^\alpha/(1+m(S+I)^\alpha)$ and $g(S,I,P)=kP^\alpha/(1+m(S+I)^\alpha)$ (Holling-type III), $f(S,I,P)=kS/(a+b(S+I)+c(S+I)^2)$ and $g(S,I,P)=kP/(a+b(S+I)+c(S+I)^2)$ (Holling-type IV), $f(S,I,P)=kS/(a+b(S+I)+cP)$ and $g(S,I,P)=kP/(a+b(S+I)+cP)$ (Beddington-De Angelis), $f(S,I,P)=kS/(a+b(S+I)+cP+d(S+I)P)$ and $g(S,I,P)=kP/(a+b(S+I)+cP+d(S+I)P)$ (Crowley-Martin). Also note that conditions in~\ref{cond:fg} are natural from a biological perspective and they are satisfied by the usual functional responses considered in the literature.

\begin{remark}
  Although it can be possible to consider a more generalised model in which some others coefficients can also be random, in this particular case we consider just one to highlight the technique and specificities of this model. For examples of bounded \emph{real noise} as considered in this model (in particular, as in~\ref{H3}) see for instance \cite{asai-kloeden-1013, caraballo-colucci-lopez-de-la-cruz-Rapaport-2019}.
\end{remark}

\subsection{Existence and properties of solutions}
In this section we prove the existence, uniqueness and boundedness of solutions to \eqref{Eq. Principal random} with nonnegative initial conditions on the populations. Moreover, we prove that the solution mapping gives rise to a RDS. We start by showing  that nonnegative initial conditions for the populations remains nonnegative, avoiding meaningless solutions in biological contexts.

\begin{lemma}\label{l:inv reg}
The set
\begin{equation}
\R_+^3=\{(S,I,P)\in \R^3:S\geq 0, I\geq 0, P\geq 0\}
\end{equation}
is positively invariant for the system \eqref{Eq. Principal random} for each fixed $\omega \in \Omega$.
\end{lemma}

\begin{proof}
The planes $ I=0 $ and $ P=0 $ are invariant since on it we have $ I'(t,\w)=0 $ and $ P'(t,\w)=0$, respectively, and $S'(t,\w)>0$ on the plane $S=0$.

If we start on the positive $P$-semi axes we have $S'(t,\w)>0$ and $I'(t,\w)=0$, so that the solution remains on $\R_+^3\cap\{I=0\}$, while if we start on the positive $I$-semi axes we have $S'(t,\w)>0$ and $P'(t,\w)=0$, so that solution does not leave $\R_+^3$.

Finally, we claim that the positive $S$-semi axes is invariant. Indeed on this semi axe we have $I'(t,\w)=P'(t,\w)=0$ and
\begin{equation}\label{Eq.S1}
S'(t,\omega)=\Lambda(\theta^t\omega)-\mu(t)S(t).
\end{equation}

That is, writing $ S_0=S(t_0,\omega)$ for the initial condition of population $S$ on time $t_0$, we have for the corresponding solution
\[
S(t;t_0,\omega,S_0)=S_0e^{-\mu(t-t_0)}+e^{-\mu t}\int_{t_0}^{t}\Lambda(\theta_s\omega)e^{\mu s}ds.
\]

Since $ \Lambda^\ell \leq \Lambda(\theta_t\omega) \leq \Lambda^u, $ the last term is bounded:
\[
\Lambda^\ell e^{-\mu t}\int_{t_0}^{t}e^{\mu s}ds \leq e^{-\mu t}\int_{t_0}^{t}\Lambda(\theta_s\omega)e^{\mu s}ds \leq \Lambda^u e^{-\mu t}\int_{t_0}^{t}e^{\mu s}ds,
\]
hence
\[
\Lambda^\ell(1-e^{-\mu(t-t_0)}) \leq \mu e^{-\mu t}\int_{t_0}^{t}\Lambda(\theta_s\omega)e^{\mu s}ds \leq \Lambda^u (1-e^{-\mu(t-t_0)}).
\]
In particular, for nonnegative initial condition $S_0$ the population $S(t;t_0,\w,S_0)$ remains nonnegative.
We conclude that the vector field at the boundary of $ \R_+^3 $ never points outwards.

\end{proof}

Set  $a^u=\max\{\gamma,r\}$, $a^\ell=\min\{\gamma,r\}$. To simplify the notation, unless stated otherwise, given $(S_0,I_0,P_0)\in \R_+^3$ we write 
$S=S(t;t_0,\omega,S_0) $, $ I=I(t;t_0,\omega,I_0) $ and $ P=P(t;t_0,\omega,P_0) $ to be the components of the solution $u(t;t_0,\omega,u_0) $ of system~\eqref{Eq. Principal random} with initial state $ u_0=(S_0,I_0,P_0) \in \mathbb{R}_+^3$ at time $t=t_0$ and fixed $\w\in\W$. Moreover, define $M_0=a^\ell S_0+r I_0+P_0$, $N_0=a^u S_0+r I_0+P_0$, 
$M=M(t;t_0,\omega,M_0)=a^\ell S+r I+P$ and $N=N(t;t_0,\omega,N_0)=a^u S+r I+P$. This notation will also be used for the particular situation $t_0=0$, which should become clear from the context. In this case we should write $
h= h(t;0,\w,h_0)=h(t;\omega,h_0),
$
for $h$ = $S$, $I$, $P$, $M$ and $N$ (and also for $h=V$ and $h=W$ to be defined later).

In the following we provide thresholds for forward invariant subsets of $\R_+^3$. Set
\[
\Theta^u=\frac{a^u\Lambda^u}{\min\{\mu,\delta_1\}},
\]
and, for $\delta\geq 0$,
\[
\Theta_\delta^\ell=\max\left\{0,\frac{a^\ell\Lambda^\ell-\delta_2(\Theta^u+\delta)^2}{\max\{c,\delta_1\}}\right\},
\]
and set  $\Theta^\ell=\Theta_0^\ell$.

\begin{proposition}\label{pr:inv region}
For each $\delta \geq 0$ the region
\begin{equation}\label{regiao_inv}\mathcal K_\delta=\left\{(S_0,I_0,P_0) \in \R_+^3: \Theta_\delta^\ell\leq M_0 \leq N_0 \leq \Theta^u+\delta\right\}
\end{equation}
is positively invariant for the system~\eqref{Eq. Principal random} for each $\omega\in\Omega$.
\end{proposition}

\begin{proof}
Let $\delta>0$ and $\omega \in \Omega$ be fixed. For $t_0 \geq 0$, we have
\begin{equation}\label{Eq:1.K}
\begin{split}
N'
& = a^u\Lambda(\theta_t\omega)-a^u\mu S-a^uf(S,I,P)P-a^u\beta SI\\
&+r\beta SI -r\eta g(S,I,P)I-rc I\\
&+\gamma f(S,I,P)P+r\eta g(S,I,P)I-\delta_1 P-\delta_2 P^2\\
& \leq a^u\Lambda^u-a^u\mu S-rc I-\delta_1 P+(\gamma -a^u)f(S,I,P)P\\
& \quad +\beta (r-a^u)SI.\\
\end{split}
\end{equation}
Since $\gamma-a^u \leq 0$ and $r-a^u \leq 0,$ we have
\begin{equation}\label{eq.C}
\begin{split}
N'\leq a^u\Lambda^u- \min\{\mu,\delta_1\}N.
\end{split}
\end{equation}
This implies
\begin{equation}\label{eq.C1}
\begin{split}
N\leq \Theta^u+ (N_0-\Theta^u)e^{-\min\{\mu,\delta_1\}(t-t_0)}.
\end{split}
\end{equation}
Similarly,
\begin{equation}
\begin{split}
M'
& \geq a^\ell\Lambda^\ell-\max\{c,\delta_1\}M\\
& \quad +(\gamma-a^\ell)f(S,I,P)P+\beta(r-a^\ell)SI\\
& \quad -\delta_2P^2. \label{Eq:2.k}
\end{split}
\end{equation}
Recall that $\gamma-a^\ell\geq0$ and $r-a^\ell\geq0$, and if $N_0\leq\Theta^u+\delta$, from~\eqref{eq.C1} we have $N\leq \Theta^u+\delta$ for all $t\geq t_0$ and
\begin{equation}\label{Eq:3.K}
\begin{split}
M'
& \geq a^\ell\Lambda^\ell-\delta_2(\Theta^u+\delta)^2 -\max\{c,\delta_1\}M,
\end{split}
\end{equation}
which implies, in this situation,
\begin{equation}\label{Eq:3.K1}
\begin{split}
M
& \geq \Theta_\delta^\ell+\left(M_0-\Theta_\delta^\ell\right)e^{-\max\{c,\delta_1\}(t-t_0)}.
\end{split}
\end{equation}
Thus if $u_0=(S_0,I_0,P_0) \in \mathcal{K}_\delta $ then $ u(t;t_0,\omega,u_0) \in \mathcal{K}_\delta \,\text{ for all } t \ge t_0 .$
\end{proof}

\begin{corollary}\label{CoralA1.1}
For all $\omega \in \Omega$, $t_0 \in \mathbb{R}_0^+$ and $(S_0, I_0, P_0) \in \mathbb{R}_+^3$ we have
\begin{equation*}
\lim_{t \to \infty}N(t;t_0,\omega,N_0) \in [\Theta^\ell, \Theta^u]\quad\text{ and }\quad   \lim_{t \to \infty}M(t;t_0,\omega,M_0) \in [\Theta^\ell, \Theta^u].
\end{equation*}
\end{corollary}

\begin{proof}
From~\eqref{eq.C1} we have
\[
\begin{split}
\lim_{t \to \infty}M(t;t_0,\omega,M_0) &\leq \lim_{t \to \infty}N(t;t_0,\omega,N_0)\\ & \leq \lim_{t \to \infty}\Theta^u+(N_0-\Theta^u)e^{-\min{\{\mu,\delta_1\}(t-t_0)}}\\ &=\Theta^u.
\end{split}
\]
In particular, for any $\delta>0$, if $t$ is sufficiently large we have $N< \Theta^u+\delta$, and from \eqref{Eq:3.K1} follows that
\[
\begin{split}
\lim_{t \to \infty}N(t;t_0,\omega,N_0) &\geq \lim_{t \to \infty}M(t;t_0,\omega,M_0)\\ & \geq \lim_{t \to \infty}\Theta_\delta^\ell+(M_0-\Theta_\delta^\ell)e^{-\max{\{c,\delta_1\}(t-t_0)}}\\&
=\Theta_\delta^\ell.\\
\end{split}
\]
\end{proof}

\begin{theorem}\label{1A}
For any $\omega \in \Omega$, $t_0 \in \mathbb{R}_0^+$ and any initial condition
$u_0=(S_0,I_0,P_0)\in \R_+^3$ the system~\eqref{Eq. Principal random} admits a unique bounded solution $u(\cdot)=u(\cdot;t_0,\omega,u_0) \in \mathcal{C}([t_0,+\infty),\R_+^3)$, with $u(t_0;t_0,\omega,u_0)=u_0$. Moreover, the solution generates a RDS $(\theta,\phi)$ defined as
\begin{equation}\label{RDE to RDS}
\phi(t,\omega,u_0)=u(t;0,\omega,u_0),\, \text{ for all }\, t \geq 0, u_0 \in \R_+^3 \, \text{and} \,\,\omega \in \Omega.
\end{equation}
\end{theorem}

\begin{proof}
The system~\eqref{Eq. Principal random}  can be rewritten in the following form
\begin{equation}\label{eq.F}
u'(t)=F(\theta_t\omega,u).
\end{equation}
Since $t\mapsto\Lambda(\theta_t\omega)$ is continuous, the map $F_\w(t,u)=F(\theta_{t}\omega,u) \in \mathcal{C}([t_0,+\infty)\times\R_+^3,\R_+^3)$ and is locally Lipschitz respect to $u$. From Corollary~\ref{CoralA1.1}, all the solutions $u(t)$ are bounded. Thus for each $\w\in\W$ the system~\eqref{Eq. Principal random} possesses a unique global solution $ u(t;t_0,\omega,u_0)$ with initial condition $u(t_0)=u_0$.

Since $ F(\theta_t\omega,u)=F(t,\w,u)$ is also measurable in $ \omega $, the map
\[u(\cdot;t_0,\cdot,\cdot):[t_0,\infty) \times \Omega \times \R_+^3  \rightarrow \R_+^3\] is $ (\mathcal{B}([t_0,\infty)) \otimes \mathcal{F} \otimes \mathcal B(\R_+^3) ,\mathcal B(\R_+^3)) $-measurable. From~\cite[Theorem 2.2.2]{Arnold}, the solutions of system~\eqref{Eq. Principal random} generate a RDS via ~\eqref{RDE to RDS}. We remark in particular the cocycle property of $(\theta,\varphi)$ that can be obtained through
\[
u(t+t_0;t_0,\omega,u_0)=u(t;0,\theta_t\omega,u_0),
\]
for all $t\geq t_0\geq0, \w\in\W$ and $u_0\in\R_+^3$.
\end{proof}

\section{Random attractors}\label{section:random_attractors}

\subsection{Deterministic global random attractor}
We establish now the existence of a pullback $\mathcal{T}(\R_+^3)$ attractor for system~\eqref{Eq. Principal random}.

\begin{theorem}
The RDS $(\theta,\phi)$ generated by~\eqref{Eq. Principal random} possesses a unique global random attractor.
\end{theorem}

The proof follows straightforward from Proposition \ref{PropA} and from the existence of a compact random absorbing set $\Gamma \in \mathcal{T}(\mathbb{R}_+^3)$ given by Proposition~\ref{PropB} below.

\begin{proposition}\label{PropB}
There exists a compact random absorbing set $\Gamma \in \mathcal{T}(\mathbb{R}_+^3)$ of the RDS $(\theta,\phi)$ generated by~\eqref{Eq. Principal random}. Moreover, for any $\delta>0$ the sets $\Gamma(\omega)$ can be chosen as the deterministic $\mathcal K_\delta$ for any $\w\in\W$.
%
\end{proposition}

\begin{proof}
Consider $A \in \mathcal{T}(\mathbb{R}_+^3)$ and  $\delta >0 $. We want to prove that for each $\omega \in \Omega$ there exists $T_A(\omega)>0$ such that  for all $t \geq T_A(\omega)$
\[
\phi(t,\theta_{-t}\omega,A(\theta_{-t}\omega)) \subseteq \mathcal{K}_\delta.
\]
From Proposition~\ref{pr:inv region} the set $\mathcal{K}_\delta$ is positively invariant, which means that for all $t\geq0$
$$\varphi(t,\w,K_\delta)\subseteq K_\delta.$$
To simplify, we write $N(t;0,\omega,  N_0)=N(t;\omega,  N_0)$ and $M(t;0,   \omega,  N_0)=M(t;\omega,  N_0)$. Recall that given $u_0=(S_0,I_0,P_0)$ we write $M_0=a^\ell S_0+rI_0+P_0$ and $N_0=a^uS_0+rI_0+P_0$.
From~\eqref{eq.C1} we have
\[
\begin{split}
N(t;\theta_{-t}\omega,  N_0) \leq \Theta^u+\sup_{u_0 \in A(\theta_{-t}\omega)} \left(N_0-\Theta^u\right)e^{-\min{\{\mu,\delta_1\}t}}.
\end{split}
\]
Since $A$ is tempered,
\begin{equation}\label{Ineqo}
\lim_{t \to \infty} \sup_{u_0 \in A(\theta_{-t}\omega)} \left( N_0-\Theta^u\right)e^{-\min{\{\mu,\delta_1\}t}}=0
\end{equation}
and thus
\begin{equation}\label{Ineq1}
\lim_{t \to \infty}N(t;\theta_{-t}\omega,  N_0) \leq \Theta^u.
\end{equation}
Assume that $\Theta_\delta^\ell>0$, otherwise the result follows from Lemma~\ref{l:inv reg}. From~\eqref{Ineq1}, for any $0<\delta'<\delta$ and $t$ sufficiently large we have $N(t;\theta_{-t}\omega,  N_0)<\Theta^u+\delta'$ and from~\eqref{Eq:3.K1}, in this situation we get
\begin{equation*}
M(t;\theta_{-t}\omega, M_0)  \geq \Theta_{\delta'}^\ell+\inf_{u_0 \in A(\theta_{-t}\omega)}\left(M_0-\Theta_{\delta'}^\ell\right)e^{-\max\{c,\delta_1\}t}.
\end{equation*}
Since
\begin{equation}
\lim_{t \to \infty} \inf_{u_0 \in A(\theta_{-t}\omega)}\left(M_0-\Theta_{\delta'}^\ell\right)e^{-\max\{c,\delta_1\}t}=0
\end{equation}
we  have
\begin{equation}\label{Ineq2}
\lim_{t \to \infty}M(t;\theta_{-t}\omega, M_0) \geq \Theta_{\delta'}^\ell>\Theta_{\delta}^\ell.
\end{equation}
Henceforth there is $T_A(\w)$ such that for all $t\geq T_A(\w)$ we have for all $u_0\in A(\theta_{-t}\w)$ that
\[\Theta_{\delta}^\ell\leq M(t;\theta_{-t}\omega, M_0)\leq N(t;\theta_{-t}\omega, N_0)\leq \Theta^u+\delta\]
and the conclusion holds.
\end{proof}

 From Remark~\ref{unique}, the global random attractor is unique. From Remark~\ref{Rm1}, $(\theta,\varphi)$ possesses a forward attractor in probability and from Remark~\ref{TD}, it also possesses {global random $\mathcal{D}(\R_+^3)$ attractor}.

\subsection{Susceptible dynamics}

 \subsubsection{Random attractor for susceptible vital dynamics}\label{S vital}
If we have no predators neither infected preys, from~\eqref{Eq. Principal random} the dynamics of susceptible preys is given by
\begin{equation}\label{eq:auxiliary-S(w)}
S'(t,\omega)=\Lambda(\theta_t\omega)-\mu S(t).
\end{equation}
For each $ \omega\in\Omega, $ the solution of~\eqref{eq:auxiliary-S(w)} with initial condition $S_0\geq 0$ at $t=t_0$ is
\[
S(t;t_0,\omega,S_0)=S_0e^{-\mu t}+\int_{t_0}^{t}\Lambda(\theta_s\omega)e^{-\mu (t-s)}ds.
\]
Replacing $\omega$ by $\theta_{-t}\omega$, and taking $t_0=0$ we have, denoting $S(t;0,\w,S_0)$ by $S(t;\w,S_0)$,
\[
S(t;\theta_{-t}\omega,S_0)=S_0e^{-\mu t}+\int_{-t}^{0}\Lambda(\theta_s\omega)e^{-\mu s}ds.
\]
For any $K\in\mathcal T([0,+\infty[)$ we have
\begin{equation*}
\lim_{t \to \infty}\sup_{S_0\in K(\theta_{-t}\w)} S_0e^{-\mu t}=0
\end{equation*}
so that we may define
\begin{equation}\label{eq.S}
\lim_{t \to \infty}S(t;\theta_{-t}\omega,S_0)= \int_{-\infty}^{0}\Lambda(\theta_s\omega)e^{-\mu s}ds:=S^*(\omega)
\end{equation}
The equation~\eqref{eq:auxiliary-S(w)} generates a RDS $(\theta,\varphi_S)$, with $ \varphi_S(t,\omega,S_0)=S(t;\omega,S_0)$, which possesses a singleton global random attractor $ \mathcal{A}(\omega)=S^*(\omega). $
Moreover, it follows from~\eqref{Lambda} and ~\eqref{eq.S} that $ S^*(\omega) \in \left[\frac{\Lambda^\ell}{\mu}, \frac{\Lambda^u}{\mu} \right]. $

 \subsubsection{Persistence of susceptible preys}\label{S per sus}
 We give conditions to ensure the prevalence of susceptible preys. We do not discuss conditions for prevalence of infected preys neither predators.
To simplify the following computations, for a given $\delta>0$ we set
\[\xi_\delta=\Lambda^\ell-f\left(\frac{\Theta^u+\delta}{a^u},0,0\right)(\Theta^u+\delta)\,\,\,\textrm{ and }\,\,\,\zeta_\delta = \mu+ \beta \left(\frac{\Theta^u+\delta}{r}\right).
\]

 \begin{proposition}
The global random attractor $\mathcal A$ for the RDS generated by~\eqref{Eq. Principal random}  possesses nontrivial  components on the $\w$-sections: $\mathcal A(\w)=(A_S(\w),A_I(\w),A_P(\w))$ with $A_S(\w)\geq \xi_\delta/\zeta_\delta$, for all $\w\in\W$. In particular, susceptible preys are prevalent if $\xi_\delta>0$ for some $\delta>0$.
\end{proposition}

\begin{proof}
 To simplify, in the following we write
\[
h= h(t;0,\theta_{-t}\w,h_0)=h(t;\theta_{-t}\omega,h_0),
\]
for $h$ = $S$, $I$ and $P$. From Proposition~\ref{PropB}, for any $\delta>0$ and any $ K \in \mathcal{T}(\R_+^3)$ there exists $T_K'(\omega) $ such that, for $t \geq T_K'(\omega)$ and $(S_0,I_0,P_0)\in K(\theta_{-t}\w)$   we have
\[ a^uS+rI+P\leq \Theta^u + \delta.
\]
From \eqref{Eq. Principal random} we therefore have
\[
\begin{split}
S'&\geq\Lambda^\ell-\mu S-f\left(\frac{\Theta^u+\delta}{a^u},0,0\right)(\Theta^u+\delta)-\beta \left(\frac{\Theta^u+\delta}{r}\right)S\\
&=\xi_\delta - \zeta_\delta S.
\end{split}
\]
Thus, for all $t \geq T_K'(\omega)$
\[
\begin{split}
S&\geq \frac{\xi_\delta}{\zeta_\delta} + \left(S(T_K'(\w);\theta_{-t}\w, S_0)-\frac{\xi_\delta}{\zeta_\delta}\right)\e^{-\zeta_\delta (t-T_K'(\omega))}.
\end{split}
\]
Hence for any $\delta>0$ and $K\in\mathcal T(\mathbb{R}_+^3)$ and large $t$, we have for all $(S_0,I_0, P_0)\in K(\theta_{-t}\w)$ 
\begin{equation}\label{eq:Spers} S=S(t;\theta_{-t}\omega,S_0)\geq \frac{\xi_\delta}{\zeta_\delta
}.
\end{equation}
\end{proof}

%

 \subsection{Extinction of predators and infected preys}
We discuss now conditions that lead to the vanish of infectious and predators.
 \begin{proposition}
The global random attractor $\mathcal A$ for the RDS generated by~\eqref{Eq. Principal random}
has singleton components $\mathcal A (\w)=(S^*(\w),0,0)$ for every $\omega \in \Omega$, provided that
\[\frac{\beta \Theta^u}{a^u}<c\quad\text{ and }\quad \gamma f\left(\frac{\Theta^u}{a^u},0, 0\right)<\delta_1.\]
\end{proposition}

\begin{proof}
The last two equations in system~\eqref{Eq. Principal random} yields to
\begin{equation}\label{eq:rI+Plinha}
\begin{split}
(rI+P)'
& = r \beta SI-rc I+\gamma f(S,I,P)P-\delta_1 P-\delta_2 P^2\\
& = (\beta S-c)rI+(\gamma f(S,I,P)-\delta_1-\delta_2 P)P.
\end{split}
\end{equation}
We will see that our hypothesis imply that both factors 
\[\beta S-c\,\,\text{ and }\,\,\gamma f(S,I,P) - \delta_1 - \delta_2 P\]
 are negative for large $t$.
If $ \frac{\beta \Theta^u}{a^u}<c $ we can choose a $ \delta'>0$ small enough such that taking $\delta=\beta\delta'/a^u$ we have
\[
\frac{\beta \Theta^u}{a^u}+\delta<c.
\]
 From Proposition~\ref{PropB} we have that $ \mathcal{K}_{\delta'}\times\Omega $ is an absorbing set in $\mathcal T(\mathbb{R}_+^3)$, so that for any $ K \in \mathcal{T}(\mathbb{R}_+^3)$ and $\omega \in \Omega$  there exists $ T_K'(\omega) $ such that for $ t \geq T_K'(\omega)$ and $(S_0,I_0,P_0)\in K(\theta_{-t}\w)$  we have
\[
\beta S=\beta S(t;\theta_{-t}\omega,S_0) \leq \frac{\beta \Theta^u}{a^u} +\delta< c,
\]
which implies that
\begin{equation}\label{IP1}
\beta S-c <0, \,\text{ for all } t \ge T_K'(\omega).
\end{equation}
Now, if $ \gamma f\left(\frac{\Theta^u}{a^u},0,0\right)<\delta_1 $, since $f$ is continuous by taking $\delta'>0$ even smaller, if necessary, we also have that
\[
\gamma f\left(\frac{\Theta^u}{a^u}+\delta,0,0 \right)<\delta_1.
\]
Again, since $ \mathcal{K}_{\delta}\times\Omega $ is also  an absorbing set in $\mathcal T(\mathbb{R}_+^3)$, for any $ K \in \mathcal{T}(\mathbb{R}_+^3)$ and $\omega \in \Omega$  there exists $ T_K(\omega)\geq T_K'(\w) $ such that, for $ t \geq T_K(\omega)$ and $(S_0,I_0,P_0)\in K(\theta_{-t}\w)$  we have $S\leq \frac{\Theta^u}{a^u}+\delta$ and, setting $P=P(t;\theta_{-t}\w,P_0)$, by Lemma \ref{l:inv reg} we have $I, P \geq 0$.
By the monotonicity of $f$,
\[
\gamma f(S,I,P) \leq \gamma f\left(\frac{\Theta^u}{a^u}+\delta,0,0 \right)\leq \delta_1,\]
which implies
\begin{equation}\label{IP2}
\gamma f(S,I,P)-\delta_1-\delta_2 P<0
\end{equation}
for all $t \ge T_K(\omega)$.
Setting $I =I(t;\theta_{-t}\w,I_0)$, from~\eqref{eq:rI+Plinha}, \eqref{IP1} and~\eqref{IP2} we have for $t\geq T_K(\w)$
\[
\begin{split}
(rI+P)'
& \leq \max\left\{\frac{\beta \Theta^u}{a^u}-c+\delta,\gamma f \left(\frac{\Theta^u}{a^u}+\delta,0,0 \right)-\delta_1\right\}(rI+P).\end{split}
\]
This implies that for all $K\in\mathcal T(\mathbb{R}_+^3)$,
\[\lim_{t\to+\infty}(rI+P)=\lim_{t\to+\infty}(rI(t;\theta_{-t}\w,I_0)+P(t;\theta_{-t}\w,P_0))=0,\]
 for all $(S_0,I_0, P_0)\in K(\theta_{-t}\w)$. Moreover, from \eqref{eq.S} if $I,P=0$ we have
\[\lim_{t\to+\infty} S(t;\theta_{-t}\omega,S_0) = S^*(\omega).
\]
Thus the global random attractor $\mathcal A$ for the RDS generated by~\eqref{Eq. Principal random} has singleton components sets $\mathcal A(\w)=\{(S^*(\w),0,0)\}$ for every $\omega \in \Omega$.
\end{proof}

\begin{example} To illustrate this result in a model we consider Holling-type I functional responses $f(S,I,P)=S$ and $g(S,I,P)=P$. Our model is in this specific case is
\begin{equation*} \label{Eq. holling type I}
\begin{cases}
S'(t,\omega)=\Lambda(\theta_t\omega)-\mu S(t)-S(t)P(t)-\beta S(t)I(t)\\
I'(t,\omega)=\beta S(t)I(t)-\eta I(t)P(t)-c I(t)\\
P'(t,\omega)=\gamma S(t)P(t)+I(t)P(t)-\delta_1 P(t)-\delta_2 P(t)^2
\end{cases}.
\end{equation*}
Thus we have simultaneous extinction of infected preys and predators, in the sense that the global random attractor has $\w$-sections of type $(S^*(\w),0,0)$, if
\[\frac{\beta \Lambda^u}{c\min\{\mu,\delta_1\}}<1\quad\text{ and }\quad  \frac{\gamma\Lambda^u}{\delta_1\min\{\mu,\delta_1\}}<1.\]
This can also be interpreted in the deterministic setting by considering $\Lambda(\w)=\Lambda_0$ for all $\w$ and some $\Lambda_0>0$.
\end{example}

\section{Random attractors for partial dynamics}\label{section:partial}
In Section~\ref{S vital} we analysed the vital dynamics of susceptible population, in the simultaneously absence of disease and predators, for which we concluded the existence of a singleton random global attractor with sections $
(S^*(\w),0,0)$. We discuss now the existence of random global attractors in other subsystems, namely either in the absence of predators or infectious, respectively. Set $\R_+^2=\{(x,y)\in\R^2\colon x,y\geq 0\}$.

\subsection{The case without predator}
Let us now consider the system~\eqref{Eq. Principal random} when we do not have predators, by making $P=0$.
 This case reduces to
\begin{equation}\label{Eq.SI random}
\begin{cases}
S'(t,\omega)=\Lambda(\theta_t\omega)-\mu S(t)-\beta S(t)I(t)\\
I'(t,\omega)=\beta S(t)I(t)-c I(t)
\end{cases}.
\end{equation}
In this situation we have a (random) epidemiological $SI$ model. This $SI$ model is a slightly different model from de $SI$ model corresponding to the first two equations ($SI$) of the $SIR$ models considered in~\cite{Caraballo-Han-2016,Caraballo-Colicci-2017}. In our work, we obtain global random attractor, prove the persistence of susceptibles and provide conditions for the extinctions of infectives.

Let us consider now $v(t;t_0,\omega,v_0)=(S(t;t_0,\omega,S_0),I(t;t_0,\omega,I_0))$ as a solution of the system~\eqref{Eq.SI random} with initial conditions $S(t_0,\omega)=S_0 $ and $I(t_0,\omega)=I_0$, and $v_0=(S_0,I_0)$. Let us  define $V=S+I$ and $V_0=S_0+I_0$. Similarly to Lemma~\ref{l:inv reg}, we easily conclude that the region $\R_2^+$ is positively invariant for system~\eqref{Eq.SI random}. In the following we provide thresholds for forward invariant subsets of $\R_+^2$.

\begin{proposition}
For each $ 0<\delta\leq\Lambda^\ell/c$ the region
\[\
\mathcal{V}_{\delta}=\left\{(S_0,I_0) \in \R_+^2:\frac{\Lambda^\ell}{c}-\delta\leq V_0 \leq \frac{\Lambda^u}{\mu}+\delta \right\}
\]
is positively invariant for the system~\eqref{Eq.SI random}.
\end{proposition}

\begin{proof}
Recall that we assume $\mu<c$. Adding the two equations in~\eqref{Eq.SI random} we have
\begin{equation}
\begin{split}
V'(t,\w)&=\Lambda(\theta_t\omega)-\mu S - cI\\
& \leq \Lambda^u-\mu V \label{eq.V}
\end{split}
\end{equation}
and
\begin{equation}
\begin{split}
V'(t,\w)
& =\Lambda(\theta_t\omega)-\mu S - cI\\
&   \geq \Lambda^\ell-c V.\label{eq.V2}
\end{split}
\end{equation}
Writing $V=V(t;0,\theta_{-t}\w,V_0)=V(t;\theta_{-t}\w,V_0)$, this implies
\begin{equation}\label{eq.V1}
\frac{\Lambda^\ell}{c}+\left(V_0-\frac{\Lambda^\ell}{c}\right)e^{-c t} \leq V \leq \frac{\Lambda^u}{\mu}+\left(V_0-\frac{\Lambda^u}{\mu}\right)e^{-\mu t}.
\end{equation}
If $(S_0,I_0)\in\mathcal V_\delta$, the solution remains in this region.

\end{proof}

We notice that from~\eqref{eq.V1} and \eqref{eq.V1} we have
\[
\begin{split}
\lim_{t \to \infty}V \leq \lim_{t \to \infty}\frac{\Lambda^u}{\mu}+(V_0-\frac{\Lambda^u}{\mu})e^{-\mu t}= \frac{\Lambda^u}{\mu}
\end{split}
\]
and
\[
\begin{split}
\lim_{t \to \infty}V \geq \lim_{t \to \infty}\frac{\Lambda^\ell}{c}+(U_0-\frac{\Lambda^\ell}{c})e^{-c t}= \frac{\Lambda^\ell}{c}
\end{split}
\]
From the previous estimates we get easily that the solutions $v$ are bounded. The following result follows straightforward as in the proof of Theorem~\eqref{1A}.

\begin{theorem}\label{SI.A}
For any $\omega \in \Omega$, $t_0 \in \mathbb{R}_0^+$ and any initial condition  $v_0=(S_0,I_0)\in \R_+^2$ the system~\eqref{Eq.SI random} admits a unique bounded solution $v(\cdot)=v(\cdot;t_0,\omega,v_0) \in \mathcal{C}([t_0,+\infty),\R_+^3$, with $v(t_0;t_0,\omega,v_0)=v_0$. Moreover, the solution generates a RDS $(\theta,\phi_{SI})$ defined as
\begin{equation}\label{RDE to RDS_SI}
\phi_{SI}(t,\omega,u_0)=v(t;0,\omega,u_0),\, \text{ for all }\, t \geq 0, v_0 \in \R_+^2 \, \text{and} \,\,\omega \in \Omega.
\end{equation}
\end{theorem}
In the following we establish the existence of a random global attractor for the partial dynamics with no predators.
\begin{theorem}
The RDS $(\theta,\phi_{SI})$ generated by \eqref{Eq.SI random} possesses a deterministic global random attractor $\mathcal A_{SI}$.
\end{theorem}

We will prove that exists a closed random absorbing set $\Gamma \in \mathcal{T}(\mathbb{R}_+^2)$. The result follows then from Proposition \ref{PropA}.

\begin{proposition}\label{PropB1}
There exists a closed random absorbing set $\Gamma \in \mathcal{T}(\mathbb{R}_+^2)$ of the RDS $(\theta,\phi_{SI})$ generated by \eqref{Eq.SI random}. Moreover, for $0<\delta \leq \frac{\Lambda^\ell}{c}$, the sets $\Gamma(\omega)$ can be chosen as the deterministic $\mathcal{V}_{\delta}$ for any $\w\in\W$.
%
\end{proposition}

\begin{proof}
Consider $A \in \mathcal{T}(\mathbb{R}_+^2)$ and $\delta >0$. We want to prove that for each $\omega \in \Omega$ there exists $T_A(\omega)>0$ such that
\[
\phi(t,\theta_{-t}\omega,A(\theta_{-t}\omega)) \subseteq \mathcal{V}_\delta\, \text{ for all }\, t \geq T_A(\omega).
\]
From~\eqref{eq.V1}, we have
\[
\begin{split}
V(t;\theta_{-t}\omega,  V_0) \leq \sup_{v_0 \in A(\theta_{-t}\omega)} \left(V_0-\frac{\Lambda^u}{\mu}\right)e^{-\mu t}+\frac{\Lambda^u}{\mu},
\end{split}
\]
and since $A$ is tempered, we have
\[
\begin{split}
\lim_{t \to \infty} \sup_{v_0 \in A(\theta_{-t}\omega)}\left(V_0-\frac{\Lambda^u}{\mu}\right)e^{-\mu t}=0
\end{split}
\]
and thus
\begin{equation}\label{Ineq.1}
\begin{split}
\lim_{t \to \infty}V(t;\theta_{-t}\omega,  V_0) \leq \frac{\Lambda^u}{\mu}.
\end{split}
\end{equation}
Similarly,
\begin{equation}\label{Ineq.2}
\begin{split}
\lim_{t \to \infty}V(t;\theta_{-t}\omega, V_0) \geq \frac{\Lambda^\ell}{c}.
\end{split}
\end{equation}
Considering the inequalities \eqref{Ineq.1} and \eqref{Ineq.2}, there exists a $T_A(\omega)$ such that for $t \geq T_A(\omega)$, we have $\phi_{SI}(t,\theta_{-t}\omega,V_0) \in \mathcal{V}_{\delta}$ for all $V_0 \in A(\theta_{-t}\omega)$. If $\delta \leq \frac{\Lambda^\ell}{c}$ then $\mathcal V_\delta\subseteq \R_+^2$.
\end{proof}

\begin{proposition}
The global random attractor $\mathcal A_{SI}$ for the RDS $(\theta,\varphi_{SI})$ generated by~\eqref{Eq.SI random}  possesses nontrivial component sets on the $\w$-sections: $\mathcal A_{SI}(\w)=(A_S(\w),A_I(\w))$ with
\[A_S(\w)\geq \frac{\Lambda^\ell}{\mu+\frac{\beta\Lambda^u}{\mu}}.\]
\end{proposition}

\begin{proof}
 Since for any $\delta>0$ the random set $\mathcal{V}_{\delta}\times\Omega$ is an absorbing set in $\mathcal T(\R_+^2)$, for any $ K \in \mathcal{T}(\R_+^3) $ there exists $ T_K(\omega) $ such that for $ t \geq T_K(\omega)$ and $(S_0,I_0)\in K(\theta_{-t}\w)$  we have
 \[
 S+I=S(t;\theta_{-t}\omega,S_0)+I(t;\theta_{-t}\omega,I_0)\leq\frac{\Lambda^u}{\mu}+\delta.
 \]
From~\eqref{Eq.SI random}~we have
\[
\begin{split}
    S'&\geq\Lambda^\ell-\mu S- \beta \left(\frac{\Lambda^u}{\mu}+\delta\right)S\\
    &=\Lambda^\ell -\left(\mu+\beta \left(\frac{\Lambda^u}{\mu}+\delta\right)\right)S.
\end{split}
\]
Hence for any $\delta>0$ and $K\in\mathcal T(\mathbb{R}_+^2)$, $(S_0,I_0)\in K(\theta_{-t}\w)$ and large $t$ we have
\[S(t;\theta_{-t}\w,S_0)\geq \frac{\Lambda^\ell}{\mu+\beta\left(\frac{\Lambda^u}{\mu}+\delta\right)}.\]
\end{proof}

In the following we give condition for an attractor without infectious component.

\begin{proposition}
The global random attractor $\mathcal A_{SI}$ for the RDS $(\theta,\varphi_{SI})$ generated by~\eqref{Eq.SI random} has singleton components $\mathcal A_{SI} (\w)=(S^*(\w),0)$ for every $\omega \in \Omega$, provided that $\frac{\beta \Lambda^u}{\mu}<c.$
\end{proposition}

\begin{proof}
From the second equation in~\eqref{Eq.SI random} we have that

\begin{equation}\label{eq.I}
I'(t,\omega)=(\beta S-c)I.
\end{equation}
Consider $ \delta>0 $ small enough such that
\[
\frac{\beta \Lambda^u}{\mu}+\beta\delta<c.
\]
From Proposition~\ref{PropB1}, for any $A\in\mathcal T(\R_+^2)$ there exists $ T_A(\omega) $ such that for all $t\geq T_A(\omega), $ and $ (S_0,I_0) \in A(\theta_{-t}\omega) $ we have
\[
\beta S =\beta S(t;\theta_{-t}\omega,S_0)\leq \frac{\beta\Lambda^u}{\mu}+\beta\delta<c
\]
which implies
\[
\beta S-c\leq\frac{\beta\Lambda^u}{\mu}+\beta\delta-c<0 ,\, \text{ for all }\,  t\geq T_A(\omega).
\]
From \eqref{eq.I} we have
\[
\lim_{t\to+\infty} I(t;\theta_{-t}\omega,I_0) \leq \lim_{t\to+\infty}  I(T_A(\w);\theta_{-t}\omega,I_0) e^{\left(\frac{\beta\Lambda^u}{\mu}-\beta\delta-c\right)(t-T_A(\w))}=0.
\]
Moreover, if $I=0$ we have from~\eqref{eq.S}   that $S(t;\theta_{-t}\omega,S_0)$ converges to $S^*(\omega)$, as $t \to +\infty$, for each $\omega \in \Omega$.
\end{proof}

\subsection*{The case without infectious}

We consider now the case that we have no infected preys in system~\eqref{Eq. Principal random}, by making $I=0$, which becomes
\begin{equation}\label{Eq.SP random}
\begin{cases}
S'(t,\omega)=\Lambda(\theta_t\omega)-\mu S-\bar f(S,P)P\\
P'(t,\omega)=\gamma \bar f(S,P)P-\delta_1 P-\delta_2 P^2
.\end{cases}
\end{equation}
where $\bar f(S,P)=f(S,0,P)$. This models corresponds to a random perturbation of a predator-prey model. We obtain global random attractor, prove the persistence of preys and provide conditions for the extinctions of predators.

Let $w(t;t_0,\omega,w_0)$ be the solution of system~\eqref{Eq.SP random} with initial condition $w_0=(S_0,P_0)$ and let $W=\gamma S+P$ and $W_0=\gamma S_0+P_0$.
Define
\[\hat\Theta^u=\frac{\gamma\Lambda^u}{\min\{\mu,\delta_1\}}
\]
and, for $\delta\geq0$
\[
\hat\Theta_\delta^\ell=\max\left\{0,\frac{\gamma \Lambda^\ell-\delta_2(\hat\Theta^u+\delta)^2}{\max\{\mu,\delta_1\}}\right\}.\]

\begin{proposition}\label{W inv}
For each $ \delta>0 $ the region
\[\
\mathcal{W}_{\delta}=\left\{(S_0,P_0) \in \R_+^2:\hat\Theta_\delta \leq W_0 \leq \hat\Theta^u+\delta \right\}
\]
is positively invariant for the system \eqref{Eq.SP random}.
\end{proposition}

\begin{proof}
From~\eqref{Eq.SP random} we have
\begin{equation}
\begin{split}
W'(t,\w)& = \gamma \Lambda(\theta_t\omega)-\gamma \mu S-\delta_1 P-\delta_2P^2\\
& \quad \leq \gamma \Lambda^u-\min\{\mu,\delta_1\}W,\label{eq.W}
\end{split}
\end{equation}
which, writing $W=W(t;0,\theta_{-t}\w,W_0)=W(t;\theta_{-t}\w,W_0)$,  implies
\begin{equation}\label{eq.WW}
W\leq \hat\Theta^u+\left(W_0-\hat\Theta^u\right)e^{-\min\{\mu,\delta_1\}t}.
\end{equation}
We also have
\begin{equation}
\begin{split}
W'(t,\w)
& \geq \gamma \Lambda^\ell-\delta_2P^2-\max\{\mu,\delta_1\}W.\label{eq.W1}
\end{split}
\end{equation}
Notice that if the initial condition $(S_0,P_0)$ belongs to $\mathcal W_\delta$ then    $P\leq  \hat\Theta^u+\delta$, and in this situation we have
\begin{equation}
\begin{split}
W'(t,\w)
& \geq \gamma \Lambda^\ell-\delta_2(\hat\Theta^u+\delta)^2-\max\{\mu,\delta_1\}W .\label{eq.W1b}
\end{split}
\end{equation}
Hence, in this case we have
\begin{equation}\label{eqWW1}
W \geq \hat\Theta_\delta^\ell+\left(W_0-\hat\Theta_\delta^\ell\right)e^{-\max\{\mu,\delta_1\}t}.
\end{equation}
If $(S_0,P_0) \in \mathcal{W}_{\delta},$ the solution remains in that region. \end{proof}

From~\eqref{eq.WW}  we have
\[
\lim_{t \to \infty}W\leq \hat\Theta^u.
\]
Notice that for any $\delta>0$, for large $t$ we have $P\leq W<\hat\Theta^u+\delta$, and then~\eqref{eqWW1} implies
\[
\lim_{t \to \infty}W\geq \hat\Theta_\delta^\ell.
\]
This implies that the solutions $w$ are bounded.

\begin{proposition}\label{1A.P}
For any $\omega \in \Omega$, $t_0 \in \mathbb{R}_0^+$ and any initial condition  $w_0=(S_0,P_0)\in (\R_0^+)^2$ the system~\eqref{Eq.SP random} admits a unique bounded solution $w(\cdot)=w(\cdot;t_0,\omega,w_0) \in \mathcal{C}([t_0,+\infty),(\R_0^+)^2)$ with $w(t_0;t_0,\omega,w_0)=w_0$.
Moreover, the solution generates a random dynamical system $(\theta,\phi_{SP})$ defined as
\[
\phi_{SP}(t,\omega,v_0)=w(t;0,\omega,w_0),\, \text{ for all }\, t \geq 0, w_0 \in (\R_0^+)^2 \, \text{and} \,\,\omega \in \Omega.
\]

\end{proposition}

\begin{theorem}
The RDS $(\theta,\phi_{SP})$ generated by \eqref{Eq.SP random} possesses a global random attractor $\mathcal A_{SP}$.
\end{theorem}

As before, the proof follows from Proposition \ref{PropA} and from the fact that there exists a closed random absorbing set $\Gamma\in \mathcal{T}(\mathbb{R}_+^2)$ given by Proposition~\ref{PropB2} below.

\begin{proposition}\label{PropB2}
There exists a closed random absorbing set $\Gamma \in \mathcal{T}(\mathbb{R}_+^2)$ of the RDS $(\theta,\phi_{SP})$ generated by~\eqref{Eq.SP random}. Moreover, for any $\delta>0$ the sets $\Gamma(\omega)$ can be chosen as the deterministic $\mathcal W_\delta$ for any $\w\in\W$.
%

\end{proposition}

\begin{proof}
Consider $A \in \mathcal{T}(\mathbb{R}_+^2)$. We will prove that for any $\omega \in \Omega$ there exists $T_A(\omega)>0$ such that
\[
\phi_{SP}(t,\theta_{-t}\omega,A(\theta_{-t}\omega) \subseteq \mathcal W_\delta \text{ for all }\, t \geq T_A(\omega).
\]
From Proposition~\ref{W inv} the region $\mathcal W_\delta$ is positively invariant for $\delta>0$.
From~\eqref{eq.WW} we have for $w_0=(S_0,P_0)\in A(\theta_{-t}\w)$ that
\[
\begin{split}
W(t;\theta_{-t}\omega,  W_0) \leq \hat\Theta^u + \sup_{w_0 \in A(\theta_{-t}\omega)} \left(W_0-\hat\Theta^u\right)e^{-\min\{\mu,\delta_1\}t}.
\end{split}
\]
Since $\mathcal A$ is tempered, we have
\[
\begin{split}
\lim_{t \to \infty} \sup_{w_0 \in A(\theta_{-t}\omega)} \left(W_0-\hat\Theta^u\right)e^{-\min\{\mu,\delta_1\}t}=0
\end{split}
\]
and thus
\begin{equation}\label{Ineq.1a}
\begin{split}
\lim_{t \to \infty}W(t;\theta_{-t}\omega,  W_0) \leq \Theta^u.
\end{split}
\end{equation}
Similarly, for any $0<\delta'<\delta$ and large $t$ we have
\[
\begin{split}
W(t;\theta_{-t}\omega,  W_0)  \geq \hat\Theta_{\delta'}^\ell+\inf_{w_0 \in A(\theta_{-t}\omega)} \left(W_0-\hat\Theta_{\delta'}^\ell\right)e^{-\max\{\mu,\delta_1\}t}.
\end{split}
\]
Since
\[
\begin{split}
\lim_{t \to \infty} \inf_{w_0 \in A(\theta_{-t}\omega)} \left(W_0-\hat\Theta_{\delta'}^\ell\right)e^{-\max\{\mu,\delta_1\}t}=0
\end{split}
\]
we have
\begin{equation}\label{Ineq.2a}
\begin{split}
\lim_{t \to \infty}W(t;\theta_{-t}\omega, W_0) \geq \Theta_{\delta'}^\ell,
\end{split}
\end{equation}
and the result follows.
\end{proof}

\begin{proposition}
The global random attractor $\mathcal A_{SP}$ for the RDS generated by~\eqref{Eq.SP random}  possesses  component sets on the $\w$-sections  $\mathcal A_{SP}(\w)=(A_S(\w),A_P(\w))$ with
\[\mu A_S(\w)\geq {\Lambda^\ell-\hat\Theta^u\bar f\left({\hat\Theta^u}/\gamma,0\right)}.\]
In particular, susceptible preys are prevalent provided the right side of inequality is positive.

\end{proposition}

\begin{proof}
 Since for any $\delta>0$ the random set $\mathcal{W}_{\delta}\times\Omega$ is an absorbing set in $\mathcal T(\R_+^2)$, for any $ K \in \mathcal{T}(\R_+^3) $ there exists $ T_K(\omega) $ such that for $ t \geq T_K(\omega)$ and $(S_0,P_0)\in K(\theta_{-t}\w)$  we have
 \[
 \gamma S+P=\gamma S(t;\theta_{-t}\omega,S_0)+P(t;\theta_{-t}\omega,P_0)\leq\hat\Theta^u+\delta.
 \]
From~\eqref{Eq.SP random} we then have for any $\delta>0$
\[
\begin{split}
    S'&>\Lambda^\ell-\mu S-\bar f\left({\hat\Theta^u}/\gamma+\delta,0\right)(\hat\Theta^u+\delta).\\
\end{split}
\]
Hence for any $\delta>0$ and $K\in\mathcal T(\mathbb{R}_+^2)$, $(S_0,I_0)\in K(\theta_{-t}\w)$ and large $t$ we have
\[S(t;\theta_{-t}\w,S_0)> \frac{\Lambda^\ell-\bar f\left({\hat\Theta^u}/\gamma+\delta,0\right)(\hat\Theta^u+\delta)}{\mu}\]
and the result follows.
\end{proof}

We give now  a condition leading to the extinction of predators.

\begin{proposition}
The global random attractor $\mathcal{A}_{SP}$ for the RDS $(\theta,\phi_{SP})$ generated by~\eqref{Eq.SP random} has a singleton components $\mathcal{A_{SP}}=(S^*(\omega),0)$  for every $\omega \in \Omega$, provided that
\[
\gamma\bar f\left(\hat\Theta^u/\gamma,0\right)<\delta_1.
\]
\end{proposition}

\begin{proof}

We can choose $ \delta>0 $ small enough, such that
\begin{equation}\label{eq:cond P to 0}\gamma \bar f\left(\frac{\hat\Theta^u}\gamma+\delta,0\right)<\delta_1.
\end{equation}
Since $\mathcal W_\delta\times\W $ is an absorbing set in $\mathcal{T}(\mathbb{R}_+^2)$, for any $K \in \mathcal{T}(\mathbb{R}_+^2),$ there is $T_K(\omega)$ such that, for $ t>0 $ sufficiently large we have
\[
S=S(t;\theta_{-t}\omega,S_0) \leq \frac{\hat\Theta^u}{\gamma}+\delta.\]
From the monotonicity of $f$ (and thus of $\bar f$) we  have
for all $t \geq T_K(\omega)$
\begin{equation}\label{P1}
\begin{split}
P'=P'(t,\theta_{-t}\omega,P_0)&=(\gamma\bar f(S,P)-\delta_1 -\delta_2 P)P\\
&\leq \left(\gamma\bar f\left({\hat\Theta^u}/{\gamma}+\delta,0\right)-\delta_1\right)P.
\end{split}
\end{equation}
From~\eqref{eq:cond P to 0} we conclude
that $P(t;\theta_{-t}\omega,P_0)$ decreases to zero as $t$ goes to infinity. Moreover, if $P=0$, from the vital dynamics of susceptible preys~\eqref{eq.S} we have for all $\w\in\W$
\[\lim_{t\to+\infty} S(t;\theta_{-t}\w,S_0)=S^*(\w).\]
\end{proof}

\bibliographystyle{elsart-num-sort}


\end{document}